\documentclass[11pt]{article}
\renewcommand{\baselinestretch}{1.2}
\newcommand{\dated}{\mbox{} \hfill {\small [{\tt \today}]}} \usepackage{amsmath,amssymb,amsfonts,diagrams}
\newenvironment{keywords}{\noindent\small {\it Keywords\/}:}{\vskip 4pt}
\newenvironment{classification}{\noindent\small 2000 {\it Mathematics Subject
Classification\/}:}{\vskip 12pt}

\newcommand{\posints}{{\mathbb N}}

\newcommand{\tensor}{\otimes}

\newcommand{\Tensor}{\hat{\otimes}}

\newcommand{\wTensor}{\check{\otimes}}

\newcommand{\cstar}{{C^\ast}}

\newcommand{\id}{{\mathrm{id}}}
\newcommand{\cb}{{\mathrm{cb}}}

\newcommand{\CB}{{\cal CB}}

\newcommand{\VN}{\operatorname{VN}}

\usepackage{amsthm,enumerate}
\theoremstyle{plain}
\newtheorem{theorem}{Theorem}[section]
\newtheorem{lemma}[theorem]{Lemma}
\newtheorem{corollary}[theorem]{Corollary}
\newtheorem{proposition}[theorem]{Proposition}
\theoremstyle{definition}
\newtheorem{definition}[theorem]{Definition}
\theoremstyle{remark}
\newtheorem*{remark}{Remark}
\newtheorem*{example}{Example}
\newtheorem*{rems}{Remarks}
\newtheorem*{exs}{Examples}
\newenvironment{remarks}{\begin{rems}\begin{enumerate}}{\end{enumerate}\end{rems}}
\newenvironment{examples}{\begin{exs}\begin{enumerate}}{\end{enumerate}\end{exs}}

\newenvironment{alphitems}{\begin{enumerate}[\rm (a)]}{\end{enumerate}}

\theoremstyle{definition}
\newtheorem*{question}{Question}
\newcommand{\AP}{\mathcal{AP}}
\newcommand{\WAP}{\mathcal{WAP}}
\newcommand{\CAP}{\mathcal{CAP}}
\title{Factorization of completely bounded maps \\ through reflexive operator spaces \\ with applications to weak almost periodicity}
\author{\textit{Volker Runde}}
\date{}
\begin{document}
\maketitle
\begin{abstract}
Let $(M,\Gamma)$ be a Hopf--von Neumann algebra, so that $M_\ast$ is a completely contractive Banach algebra. We investigate whether the product of two elements of $M$ that are both weakly almost periodic functionals on $M_\ast$ is again weakly almost periodic. For that purpose, we establish the following factorization result: If $M$ and $N$ are injective von Neumann algebras, and if $\boldsymbol{x}, \boldsymbol{y} \in M \bar{\tensor} N$ correspond to weakly compact operators from $M_\ast$ to $N$ factoring through reflexive operator spaces $X$ and $Y$, respectively, then the operator corresponding to $\boldsymbol{xy}$ factors through the Haagerup tensor product $X \tensor^h Y$ provided that $X \tensor^h Y$ is reflexive. As a consequence, for instance, for any Hopf--von Neumann algebra $(M,\Gamma)$ with $M$ injective, the product of a weakly almost periodic element of $M$ with a completely almost periodic one is again weakly almost periodic.
\end{abstract}
\begin{keywords}
factorization; Haagerup tensor product; Hopf--von Neumann algebra; reflexive operator space; weakly compact map.
\end{keywords}
\begin{classification}
Primary 47L25; Secondary 22D25, 43A30, 46B10, 46B28, 46L06, 46L07, 47B07.
\end{classification}
\section*{Introduction}
Given a locally compact group $G$, a bounded, continuous function on $G$ is called \emph{almost periodic} or \emph{weakly almost periodic}, respectively, if the set of its left translates is relatively norm or weakly compact, respectively, in the space $\mathcal{C}(G)$ of all bounded, continuous functions on $G$. The spaces $\AP(G)$ and $\WAP(G)$ of almost and weakly almost periodic functions, respectively, are well known to be unital $\cstar$-subalgebras of $\mathcal{C}(G)$ (\cite{Bur} or \cite{BJM})).
\par 
The concepts of almost and weak almost periodicity can be dealt with in a more abstract framework. If $A$ is a Banach algebra, its dual space $A^\ast$ is a Banach $A$-bimodule in a canonical fashion, and a functional $\phi \in A^\ast$ is called \emph{almost} or \emph{weakly almost periodic}, respectively, if $\{ a \cdot \phi : a \in A, \, \| a \| \leq 1 \}$ is relatively norm or weakly compact, respectively, in $A^\ast$; the sets of almost and weakly almost functionals of $A$ are closed, linear subspaces of $A^\ast$ and denoted by $\AP(A)$ and $\WAP(A)$, respectively. For $A = L^1(G)$, $\AP(A)$ and $\WAP(G)$ are just $\AP(G)$ and $\WAP(G)$ (\cite{Ulg}).
\par 
If $A$ is Eymard's Fourier algebra $A(G)$ (\cite{Eym}), then $\AP(A)$ and $\WAP(A)$ are commonly denoted by $\AP(\hat{G})$ and $\WAP(\hat{G})$, respectively. It is easy to see that both $\AP(\hat{G})$ and $\WAP(\hat{G})$ are self-adjoint subspaces of $\VN(G)$, the group von Neumann algebra of $G$, containing the identity. But except in a few fairly obvious cases---for abelian $G$ by Pontryagin duality or for discrete and amenable $G$ by \cite[Proposition 3(b)]{Gra}---, it has been unknown to this day whether or not $\AP(\hat{G})$ and $\WAP(\hat{G})$ are $\cstar$-subalgebras of $\VN(G)$.
\par 
There is a common framework, somewhat less general than that of general Banach algebras, to study $\AP(G)$, $\WAP(G)$, $\AP(\hat{G})$, and $\WAP(\hat{G})$, namely that of Hopf--von Neumann algebras (see \cite{ES}, for instance). Given a Hopf--von Neumann algebra $(M,\Gamma)$, the predual $M_\ast$ is canonically equipped with a multiplication turning it into a completely contractive Banach algebra: both $L^1(G)$ and $A(G)$ are Banach algebras arising in this fashion. The question of whether $\AP(\hat{G})$ and $\WAP(\hat{G})$ are $\cstar$-subalgebras of $\VN(G)$ is therefore just a special case of the more general problem whether $\AP(M_\ast)$ and $\WAP(M_\ast)$ are $\cstar$-subalgebras of $M$ for every Hopf--von Neumann algebra $(M,\Gamma)$.
\par 
Recently, some progress was achieved towards a solution of this problem. For instance, M.\ Daws (see \cite{Daw2}) showed that, if $M$ is abelian, then both $\AP(M_\ast)$ and $\WAP(M_\ast)$ are $\cstar$-subalgebras of $M$. Unfortunately, as the author was able to show in \cite{Run1}, the methods used by Daws to prove that $\WAP(M_\ast)$ is a $\cstar$-algebra cannot be extended beyond subhomogeneous von Neumann algebras. Still, Daws' results entail that both $\AP(M(G))$ and $\WAP(M(G))$ are $\cstar$-subalgebras of the commutative von Neumann algebra $\mathcal{C}_0(G)^{\ast\ast}$. Furthermore, in \cite{Run2}, the author used the notion of complete compactness---as introduced by H.\ Saar in \cite{Saa}---to introduce the notion of a \emph{completely almost periodic} functional on a completely contractive Banach algebra; unlike ordinary almost periodicity, complete almost periodicity takes operator space structures into account. The main result of \cite{Run2} asserts that, if $(M,\Gamma)$ is a Hopf--von Neumann algebra such that $M$ is injective, then the space $\CAP(M_\ast)$ of all completely almost periodic functionals on $M_\ast$ is a $\cstar$-subalgebra of $M$.
\par
The present paper is motivated by the question of when, for a Hopf--von Neumann algebra $(M,\Gamma)$, the closed, self-adjoint subspace $\WAP(M_\ast)$ of $M$ is closed under multiplication (and thus a $\cstar$-subalgebra of $M$).
\par
If $M$ and $N$ are von Neumann algebras, then each element $\boldsymbol{x} \in M \bar{\tensor} N$ corresponds in a one-to-one fashion to a completely bounded operator $\mathcal{T}_{M,N}\boldsymbol{x}$ from $M_\ast$ to $N$, namely $N_\ast \ni f \mapsto (f \tensor \id)(\boldsymbol{x})$. For a Hopf--von Neumann algebra $(M,\Gamma)$, it is easy to see that
\[
  \WAP(M_\ast) = \{ x \in M : \text{$\mathcal{T}_{M,M}(\Gamma x)$ is weakly compact} \}.
\]
We are thus interested in whether, for $\boldsymbol{x}, \boldsymbol{y} \in M \bar{\tensor} N$ with $\mathcal{T}_{M,N}\boldsymbol{x}$ and $\mathcal{T}_{M,N}\boldsymbol{y}$ weakly compact, $\mathcal{T}_{M,N}(\boldsymbol{xy})$ is weakly compact. In analogy with the Banach space situation (\cite{Davetal}), a completely bounded map is weakly compact if and only if it factors through a reflexive operator space (\cite{PSch} or \cite{Daw1}). Thus, our main tool for tackling this question is the following factorization result: If $M$ and $N$ are injective von Neumann algebras, and if $\boldsymbol{x},\boldsymbol{y} \in M \bar{\tensor} N$ are such that $\mathcal{T}_{M,N}\boldsymbol{x}$ and $\mathcal{T}_{M,N}\boldsymbol{y}$ factor through operator spaces $X$ and $Y$, respectively, with $X \tensor^h Y$---their Haagerup tensor product---reflexive, then $\mathcal{T}_{M,N}(\boldsymbol{xy})$ factors through $X \tensor^h Y$ (and, consequently, is weakly compact). Even though the Haagerup tensor product of two reflexive operator spaces may well fail to be compact, this allows for some interesting insights
\par 
Applying our findings to weakly almost periodic elements in Hopf--von Neumann algebras, we recover for instance (and even improve slightly) the main result of \cite{Daw2}, and we show that, if $(M,\Gamma)$ is a Hopf--von Neumann algebra with $M$ injective, then $xy, yx \in \WAP(M_\ast)$ for any $x \in \WAP(M_\ast)$ and $y \in \CAP(M_\ast)$.
\subsubsection*{Acknowledgment}
I would like to thank Matt Daws for carefully reading an earlier version of this paper and David Blecher for bringing \cite[Theorem 3.1]{Ble} to my attention.
\section{A factorization result for completely bounded maps}
Our reference for operator spaces is \cite{ER}, the notation of which we adopt; in particular, $\Tensor$ stands for the projective tensor product of operator spaces, not of Banach spaces (see \cite[Chapter 7]{ER}).
\par 
Given two operator spaces $E$ and $F$, there are two canonical ways of looking at the dual of their projective tensor product $E \Tensor F$. On the one hand, we can completely isometrically identify $(E \Tensor F)^\ast$ with $\CB(E,F^\ast)$, the space of all completely bounded maps from $E$ into $F^\ast$ (\cite[Corollary 7.1.5]{ER}). There is, however, another way to describe $(E \Tensor F)^\ast$. 
\par 
Recall that there are are Hilbert space $H$ and $K$ such that $E^\ast$ and $F^\ast$ have dual realizations on $H$ and $K$, respectively (\cite[Proposition 3.2.4]{ER}), i.e., there are weak$^\ast$ continuous complete isometries $E^\ast \hookrightarrow \mathcal{B}(H)$ and $F^\ast \hookrightarrow \mathcal{B}(K)$. The \emph{normal spatial tensor product} $E^\ast \bar{\tensor} F^\ast$ of $E^\ast$ and $F^\ast$ is defined as the weak$^\ast$ closure of the algebraic tensor product $E \tensor F$ in $\mathcal{B}(H \tensor_2 K)$, where $\tensor_2$ denotes the Hilbert space tensor product (\cite[p.\ 134]{ER}). Note that $\mathcal{B}(H) \bar{\tensor} \mathcal{B}(K) = \mathcal{B}(H \tensor_2 K)$. Given $\nu \in \mathcal{B}(H)_\ast$, the predual of $\mathcal{B}(H)$, we denote by $\nu \tensor \id$, the corresponding \emph{Tomiyama slice map}, i.e., the unique weak$^\ast$-weak$^\ast$ continuous extension of $\mathcal{B}(H) \tensor \mathcal{B}(H) \ni x \tensor y \mapsto \langle \nu, x \rangle y$; similarly, $\id \tensor \omega$ is defined for $\omega \in \mathcal{B}(K)_\ast$. The \emph{normal Fubini tensor product} of $E^\ast$ and $F^\ast$ (see \cite[p.\ 134]{ER}) is defined to be
\begin{multline*}
  E^\ast \bar{\tensor}_\mathcal{F} F^\ast := \{ \boldsymbol{t} \in \mathcal{B}(H \tensor_2 K) : \text{$(\nu \tensor \id)(\boldsymbol{t}) \in F^\ast$ and $(\id \tensor \omega)(\boldsymbol{t}) \in F^\ast$} \\ \text{for all $\nu \in \mathcal{B}(H)_\ast$ and $\omega \in \mathcal{B}(K)_\ast$} \}.
\end{multline*}
By \cite[Theorem 7.2.3]{ER}, we have a canonical completely isometric isomorphism between $(E \Tensor F)^\ast$ and $E^\ast \bar{\tensor}_\mathcal{F} F^\ast$, so that, in particular, $E^\ast \bar{\tensor}_\mathcal{F} F^\ast$ does not depend on the particular dual realizations of $E^\ast$ and $F^\ast$, respectively (as is the case for $E^\ast \bar{\tensor} F^\ast$ by \cite[Proposition 8.1.8]{ER}).
\par 
In view of the two ways to realize $(E \Tensor F)^\ast$, we thus have a completely isometric isomorphism $\mathcal{T}_{E^\ast,F^\ast} \!: E^\ast \bar{\tensor}_\mathcal{F} F^\ast \to \CB(E,F^\ast)$, given by
\begin{equation} \label{firstiso}
  (\mathcal{T}_{E^\ast,F^\ast} \boldsymbol{t})x = (x \tensor \id)(\boldsymbol{t}) \qquad (\boldsymbol{t} \in E^\ast \bar{\tensor}_\mathcal{F} F^\ast, \, x \in E). 
\end{equation}
\par 
Clearly, $E^\ast \bar{\tensor} F^\ast$ is a closed subspace of $E^\ast \bar{\tensor}_\mathcal{F} F^\ast$, and both spaces coincide, for instance, if both $E^\ast$ and $F^\ast$ are von Neumann algebras (this follows from \cite[Theorem 7.2.4]{ER}). In general, $E^\ast \bar{\tensor} F^\ast$ may be a proper subspace of $E^\ast \bar{\tensor}_\mathcal{F} F^\ast$---even if one of $E^\ast$ and $F^\ast$ is a von Neumann algebra (\cite[Theorem 3.3]{Kra2}). Following \cite{Kra1}, we say that $E^\ast$ has \emph{property $S_\sigma$} if $E^\ast \bar{\tensor} F^\ast = E^\ast \bar{\tensor}_\mathcal{F} F^\ast$ for any choice of $F$. Injective von Neumann algebras, for instance, have property $S_\sigma$ (\cite[Theorem 1.9]{Kra1}).
\par 
The \emph{Haagerup tensor product} $\tensor^h$ is defined and discussed in \cite[Chapter 9]{ER}. For the related notions of the \emph{extended Haagerup product} $\tensor^{eh}$ and the \emph{normal Haagerup tensor product} $\tensor^{\sigma h}$---introduced in \cite{EK}---, see \cite{ERJoT}. In \cite[Theorem 6.1]{ERJoT}, the authors relate $\bar{\tensor}$ and $\tensor^{\sigma h}$ by showing that, for any operator spaces $E_1$, $F_1$, $E_2$, and $F_2$ the shuffle map
\[
  \mathcal{S} \!: (E_1^\ast \tensor F^\ast_1) \tensor (E_2^\ast \tensor F_2^\ast) \to (E_1^\ast \tensor E_2^\ast) \tensor (F_1^\ast \tensor F_2^\ast) 
\]
has---necessarily unique---weak$^\ast$-weak$^\ast$ continuous, completely contractive extension
\[
  \mathcal{S}_\sigma \!: (E_1^\ast \bar{\tensor} F^\ast_1) \tensor^{\sigma h} (E_2^\ast \bar{\tensor} F_2^\ast) \to (E_1^\ast \tensor^{\sigma h} E_2^\ast) \bar{\tensor} (F_1^\ast \tensor^{\sigma h} F_2^\ast) 
\]
\par
Before we can finally state the main result of this section, we introduce another convention: given operator spaces $E$, $F$, and $X$, we say that $T \in \CB(E,F)$ \emph{factors completely boundedly through $X$} if there are $R \in \CB(E,X)$ and $S \in \CB(X,F)$ such that $T = SR$. (For the sake of brevity, we will sometimes drop the words ``completely boundedly'' if no confusion can arise.)
\begin{theorem} \label{mainfactorthm}
Let $E_1$, $F_1$, $X_1$, $E_2$, $F_2$, and $X_2$ be operator spaces such that:
\begin{alphitems}
\item $E_1^\ast$, $F_1^\ast$, $E_2^\ast$, and $F_2^\ast$ have property $S_\sigma$;
\item $X_1 \tensor^h X_2$ is reflexive.
\end{alphitems}
Furthermore, let $\boldsymbol{t}_j \in E^\ast_j \bar{\tensor} F^\ast_j$ be such that $\mathcal{T}_{E_j^\ast, F_j^\ast} \boldsymbol{t}_j$ factors completely boundedly through $X_j$ for $j=1,2$. Then 
\begin{equation} \label{monster}
  \mathcal{T}_{E_1^\ast \tensor^{\sigma h} E_2^\ast, F_1^\ast \tensor^{\sigma h} F_2^\ast} (\mathcal{S}_\sigma(\boldsymbol{t}_1 \tensor \boldsymbol{t}_2)) \in \CB(E_1 \tensor^{eh} E_2, F^\ast_1 \tensor^{\sigma h} F_2^\ast)
\end{equation}
factors completely boundedly through $X_1 \tensor^h X_2$.
\end{theorem}
\par 
Before we start proving Theorem \ref{mainfactorthm}, we would like to comment on our choice of hypotheses, especially Theorem \ref{mainfactorthm}(b). 
\par 
The reflexivity of $X_1 \tensor^h X_2$ forces both $X_1$ and $X_2$ to be reflexive because $X_1 \tensor^h X_2$ contains isomorphic copies of both $X_1$ and $X_2$. Consequently, $\mathcal{T}_{E_j^\ast, F_j^\ast} \boldsymbol{t}_j$ is weakly compact for $j=1,2$ as is (\ref{monster}). 
\par 
It is a classical result---from \cite{Davetal}---that every weakly compact operator between Banach spaces factors through a reflexive Banach space. The analogous statement is true in the category of operator spaces (\cite{PSch}): any completely bounded, weakly compact map between operator spaces factors completely boundedly through a reflexive operator space. (Apparently without knowledge of \cite{PSch}, Daws discovered a similar result, which provides better norm estimates and also takes module structures into account; see \cite[Theorem 4.4]{Daw1}.) 
\par 
In view of this, Theorem \ref{mainfactorthm} would be much more attractive if $X_1$ and $X_2$ being reflexive entailed the reflexivity of $X_1 \tensor^h X_2$; in certain cases, this is indeed true, but not always:
\begin{examples}
\item Suppose that $X_1$ and $X_2$ are reflexive with one of them finite-dimensional. Then $X_1 \tensor^h X_2$ is trivially reflexive.
\item Suppose that $X_1$ is a minimal and that $X_2$ is a maximal operator space. Then $X_1 \tensor^h X_2$ is reflexive by \cite[Theorem 3.1(v)]{Ble} and \cite{AS}. Similarly, $X_1 \tensor^h X_2$ is also reflexive if $X_1$ is maximal and $X_2$ is minimal.
\item More generally, suppose that $X_1$ and $X_2$ are reflexive, $X_1$ is minimal on its rows, and $X_2$ is maximal on its columns, i.e.,
\begin{multline*}
  \| x_1 \|_{M_{1,n}(X_1)} = \| x \|_{M_{1,n}(\min X_1)} \\ \quad\text{and}\quad \| x_2 \|_{M_{n,1}(X_2)} = \| x \|_{M_{n,1}(\max X_2)} \qquad (n \in \posints, x_1 \in M_n(X_1), \, x_2 \in X_2).
\end{multline*}
As only the rows of $X_1$ and the columns of $X_2$ are relevant for the definition of $X_1 \tensor^h X_2$ at the Banach space level, the previous example yields the reflexivity of $X_1 \tensor^h X_2$. In \cite{Lam} (see also \cite{LNR}), A.\ Lambert defined an operator space---called \emph{column operator space} in \cite{Lam}---over an arbitrary Banach space which is indeed maximal on its columns and minimal on its rows (for Hilbert spaces, this is just the usual column Hilbert space by \cite{Mat}). Similarly, $X_1 \tensor^h X_2$ is also reflexive if $X_1$ and $X_2$ are reflexive with $X_1$ maximal on its rows and $X_2$ minimal on its columns.
\item Let $H$ and $K$ be Hilbert spaces, and suppose that $X_1 = H_c$, i.e., column Hilbert space over $H$, and $X_2 = (K_c)^\ast$. Then $X_1$ and $X_2$ are obviously reflexive whereas $X_1 \tensor^h X_2 \cong \mathcal{K}(K,H)$ (\cite[Proposition 9.3.4]{ER}), i.e., the compact operators from $K$ to $H$ isn't unless $X_1$ or $X_2$ is finite-dimensional.
\end{examples}
\par 
As we just noted, $X_1 \tensor^h X_2$ need not be reflexive, even if both $X_1$ and $X_2$ are. \emph{If} they are reflexive, however, we have:
\begin{lemma} \label{hrefl}
Let $X_1$ and $X_2$ be operator spaces such that $X_1 \tensor^h X_2$ is reflexive. Then the canonical completely isometric maps
\[
  X_1 \tensor^h X_2 \hookrightarrow X_1 \tensor^{eh} X_2 \hookrightarrow X_1 \tensor^{\sigma h} X_2
\]
are onto, and there is a canonical completely isometric isomorphism between $(X_1 \tensor^h X_2)^\ast$ and $X_1^\ast \tensor^h X_2^\ast$.
\end{lemma}
\begin{proof}
From \cite[Theorem 5.3]{ERJoT} and the definition of $\tensor^{\sigma h}$, it follows that $X_1 \tensor^h X_2 \hookrightarrow X_1 \tensor^{\sigma h} X_2$ is nothing but the canonical embedding of $X_1 \tensor^h X_2$ into its second dual. The reflexivity of $X_1 \tensor^h X_2$ thus yields that $X_1 \tensor^h X_2 \hookrightarrow X_1 \tensor^{\sigma h} X_2$ is onto. Consequently, $X_1 \tensor^h X_2 \hookrightarrow X_1 \tensor^{eh} X_2$ then also has to be onto.
\par
If $X_1 \tensor^h X_2$ is reflexive, then so is its dual---by \cite[Theorem 5.3]{ERJoT}---$X_1^\ast \tensor^{eh} X_2^\ast$. Since $X_1^\ast \tensor^h X_2^\ast$ embeds canonically into $X_1^\ast \tensor^{eh} X_2^\ast$, this means that $X_1^\ast \tensor^h X_2^\ast$ must also be reflexive, so that $(X_1 \tensor^h X_2)^\ast \cong X_1^\ast \tensor^{eh} X_2^\ast = X_1^\ast \tensor^h X_2^\ast$ because $X_1^\ast \tensor^h X_2^\ast \hookrightarrow X_1^\ast \tensor^{eh} X_2^\ast$ has to be onto by the first part of the lemma.
\end{proof}
\begin{proof}[Proof of Theorem \emph{\ref{mainfactorthm}}]
As $\mathcal{T}_{E_j^\ast, F_j^\ast} \boldsymbol{t}_j$ factors completely boundedly through $X_j$ for $j=1,2$, there are $\boldsymbol{r}_j \in E_j^\ast \bar{\tensor}_\mathcal{F} X_j$ and $\boldsymbol{s}_j \in X_j^\ast \bar{\tensor}_\mathcal{F} F_j^\ast$ such that
\[
  \mathcal{T}_{E^\ast_j,F_j^\ast} \boldsymbol{t}_j = (\mathcal{T}_{X_j^\ast,F^\ast_j} \boldsymbol{s}_j)(\mathcal{T}_{E_j^\ast,X_j} \boldsymbol{r}_j) \qquad (j =1,2).
\]
Note that, by hypothesis (a), we have $E_j^\ast \bar{\tensor}_\mathcal{F} X_j = E_j^\ast \bar{\tensor} X_j$ and $X_j^\ast \bar{\tensor}_\mathcal{F} F_j^\ast = X_j^\ast \bar{\tensor} F_j^\ast$, so that, in fact, $\boldsymbol{r}_j \in E_j^\ast \bar{\tensor} X_j$ and $\boldsymbol{s}_j \in X_j^\ast \bar{\tensor} F_j^\ast$ for $j=1,2$. We claim that
\begin{multline*}
 \mathcal{T}_{E_1^\ast \tensor^{\sigma h} E_2^\ast, F_1^\ast \tensor^{\sigma h} F_2^\ast} (\mathcal{S}_\sigma(\boldsymbol{t}_1 \tensor \boldsymbol{t}_2)) \\ 
 = \mathcal{T}_{X_1^\ast \tensor^h X_2^\ast, F_1^\ast \tensor^{\sigma h} F_2^\ast} (\mathcal{S}_\sigma(\boldsymbol{s}_1 \tensor \boldsymbol{s}_2)) \mathcal{T}_{E_1^\ast \tensor^{\sigma h} E_2^\ast, X_1 \tensor^h X_2} (\mathcal{S}_\sigma(\boldsymbol{r}_1 \tensor \boldsymbol{r}_2))
\end{multline*}
\par 
Define a bilinear map
\[
  \mu_1 \!: (E_1^\ast \tensor X_1) \times (X_1^\ast \tensor F^\ast_1) \to E_1^\ast \tensor F_1^\ast, \quad
  ((e^\ast \tensor x), (x^\ast \tensor f^\ast)) \mapsto \langle x, x^\ast \rangle e^\ast \tensor f^\ast.
\]
It is immediate that
\begin{equation} \label{mu1eq}
  \mathcal{T}_{E_1^\ast, F_1^\ast} (\mu_1(\boldsymbol{r}, \boldsymbol{s})) = (\mathcal{T}_{X_1^\ast,F_1^\ast} \boldsymbol{s}) (\mathcal{T}_{E_1^\ast,X_1} \boldsymbol{r}) 
\end{equation}
for all $\boldsymbol{s} \in X_1^\ast \tensor F_1^\ast$ and $\boldsymbol{r} \in E_1^\ast \tensor X_1$. As the composition map from $\CB(E_1,X_1) \times \CB(X_1,F_1^\ast)$ into $\CB(E_1,F_1^\ast)$ is separately weak$^\ast$-weak$^\ast$ continuous---due to the reflexivity of $X_1$---, we see that $\mu_1$ has---necessarily unique---separately weak$^\ast$-weak$^\ast$ continuous extension from $(E_1^\ast \bar{\tensor} X_1) \times (X_1^\ast \bar{\tensor} F^\ast_1)$ to $E_1^\ast \bar{\tensor} F_1^\ast$, which we also denote by $\mu_1$. Clearly, (\ref{mu1eq}) then also holds for all $\boldsymbol{s} \in X_1^\ast \bar{\tensor} F_1^\ast$ and $\boldsymbol{r} \in E_1^\ast \bar{\tensor} X_1$. Analogously, we define
\[
  \mu_2 \!: (E_2^\ast \bar{\tensor} X_2) \times (X_2^\ast \bar{\tensor} F^\ast_2) \to E_2^\ast \bar{\tensor} F_2^\ast
\]
and 
\[
  \mu_\sigma \!: ((E_1^\ast \tensor^{\sigma h} E_2^\ast) \bar{\tensor} (X_1 \tensor^h X_2)) \times ((X_1^\ast \tensor^h X_2^\ast) \bar{\tensor} (F_1^\ast \tensor^{\sigma h} F^\ast_2)) \to (E_1^\ast \tensor^{\sigma h} E_2^\ast) \bar{\tensor} (F_1^\ast \tensor^{\sigma h} F_2^\ast).
\]
We are done if the diagram
\begin{small}
\begin{equation} \label{diagram}
  \begin{diagram}
  (E_1^\ast \bar{\tensor} X_1) \times (X_1^\ast \bar{\tensor} F^\ast_1) \times (E_2^\ast \bar{\tensor} X_2) \times (X_2^\ast \bar{\tensor} F^\ast_2)
  & \rTo^{\mu_1 \times \mu_2} & (E_1^\ast \bar{\tensor} F_1^\ast) \tensor^{\sigma h} (E_2^\ast \bar{\tensor} F_2^\ast) \\
  \dTo & & \\
  ((E_1^\ast \bar{\tensor} X_1) \tensor^{\sigma h} (E_2^\ast \bar{\tensor} X_2)) \times ((X_1^\ast \bar{\tensor} F_1^\ast) \tensor^{\sigma h} (X_2^\ast \bar{\tensor} F^\ast_2)) & & \dTo_{\mathcal{S}_\sigma} \\
  \dTo^{\mathcal{S}_\sigma \times \mathcal{S}_\sigma} & &  \\ 
  ((E_1^\ast \tensor^{\sigma h} E_2^\ast) \bar{\tensor} (X_1 \tensor^h X_2)) \times ((X_1^\ast \tensor^h X_2^\ast) \bar{\tensor} (F_1^\ast \tensor^{\sigma h} F_2^\ast)) & \rTo_{\mu_\sigma} &
  (E_1^\ast \tensor^{\sigma h} E_2^\ast) \bar{\tensor} (F_1^\ast \tensor^{\sigma h} F_2^\ast)
  \end{diagram}
\end{equation}
\end{small}
commutes. 
\par 
From the definitions of the maps involved, it is immediate that (\ref{diagram}) commutes if restricted to $(E_1^\ast \tensor X_1) \times (X_1^\ast \tensor F^\ast_1) \times (E_2^\ast \tensor X_2)$. The commutativity of (\ref{diagram}) the follows from the (separate) weak$^\ast$-weak$^\ast$ continuity of the maps in (\ref{diagram})
\end{proof}
\section{Elements of von Neumann algebra tensor products corresponding to weakly compact operators}
Let $M$ and $N$ be von Neumann algebras. We set
\[
  \mathcal{W}(M \bar{\tensor} N) := \{ \boldsymbol{x} \in M \bar{\tensor} N : \text{$\mathcal{T}_{M,N} \boldsymbol{x} \in \CB(M_\ast, N)$ is weakly compact} \}.
\]
It is easy to see that $\mathcal{W}(M \bar{\tensor} N)$ is a closed, self-adjoint subspace of $M \bar{\tensor} N$ containing the identity. We are interested in the question if $\mathcal{W}(M \bar{\tensor} N)$ a $\cstar$-subalgebra of $M \bar{\tensor} N$. Of course, all that needs verification is whether $\mathcal{W}(M \bar{\tensor} N)$ is multiplicatively closed, i.e., if $\boldsymbol{x}, \boldsymbol{y} \in \mathcal{W}(M \bar{\tensor} N)$, is then $\boldsymbol{xy} \in \mathcal{W}(M \bar{\tensor} N)$? 
\par 
We cannot answer this question in general, but we shall obtain some partial results as applications of the following theorem, which is a consequence of Theorem \ref{mainfactorthm}:
\begin{theorem} \label{vNthm}
Let $M$ and $N$ be injective von Neumann algebras, let $X$ and $Y$ be operator spaces such that $X \tensor^h Y$ is reflexive, and let $\boldsymbol{x}, \boldsymbol{y} \in M \bar{\tensor} N$ be such that:
\begin{alphitems}
\item $\mathcal{T}_{M,N} \boldsymbol{x}$ factors completely boundedly through $X$;
\item $\mathcal{T}_{M,N} \boldsymbol{y}$ factors completely boundedly through $Y$.
\end{alphitems}
Then $\mathcal{T}_{M,N}(\boldsymbol{xy})$ factors completely boundedly through $X \tensor^h Y$.
\end{theorem}
\begin{proof}
First, note that, by Theorem \ref{mainfactorthm}, the operator $\mathcal{T}_{M \tensor^{\sigma h} M, N \tensor^{\sigma h} N}(\mathcal{S}_\sigma (\boldsymbol{x} \tensor \boldsymbol{y})) \in \CB(M_\ast \tensor^{eh} M_\ast, N \tensor^{\sigma h} N)$ factors through $X \tensor^h Y$.
\par 
Multiplication in $M$ is a separately weak$^\ast$ continuous, multiplicatively bounded, bilinear map and thus induces a unique weak$^\ast$ continuous complete contraction $m_M \!: M \tensor^{\sigma h} M \to M$ (\cite[Proposition 5.9]{ERJoT}). Analogously, we obtain $m_N \!: N \tensor^{\sigma h} N \to N$ and $m_{M \bar{\tensor} N} \!: (M \bar{\tensor} N) \tensor^{\sigma h} (M \bar{\tensor} N) \to M \bar{\tensor} N$. We claim that
\begin{equation} \label{faceq}
  \mathcal{T}_{M,N}(\boldsymbol{xy}) = m_N \, \mathcal{T}_{M \tensor^{\sigma h} M, N \tensor^{\sigma h} N}(\mathcal{S}_\sigma (\boldsymbol{x} \tensor \boldsymbol{y})) \, (m_M)_\ast,
\end{equation}
where $(m_M)_\ast \!: M_\ast \to M_\ast \tensor^{eh} M_\ast$ is the preadjoint of $m_M$. Of course, if (\ref{faceq}) holds, the theorem is proven.
\par 
Let $\mu \!: (M \bar{\tensor} (X \tensor^h Y)) \times ((X^\ast \tensor^h Y^\ast) \bar{\tensor} N) \to M \bar{\tensor} N$ and $\mu_\sigma \!: ((M \tensor^{\sigma h} M) \bar{\tensor} (X \tensor^h Y)) \times ((X^\ast \tensor^h Y^\ast) \bar{\tensor} (N \tensor^{\sigma h} N)) \to (M \tensor^{\sigma h} M) \bar{\tensor} (N \tensor^{\sigma h} N)$ be the respective composition maps as in the proof of Theorem \ref{mainfactorthm} and consider the diagram
\begin{equation} \label{diagram2}
\begin{diagram}
  ((M \tensor^{\sigma h} M) \bar{\tensor} (X \tensor^h Y)) \times ((X^\ast \tensor^h Y^\ast) \bar{\tensor} (N \tensor^{\sigma h} N)) & \rTo^{\mu_\sigma} & (M \tensor^{\sigma h} M) \bar{\tensor} (N \tensor^{\sigma h} N) \\
  \dTo^{(m_M \tensor \id) \times (\id \tensor m_N)} & & \dTo_{m_{M \bar{\tensor} N}} \\
  (M \bar{\tensor} (X \tensor^h Y)) \times ((X^\ast \tensor^h Y^\ast) \bar{\tensor} N) & \rTo_{\mu} & M \bar{\tensor} N.
\end{diagram}
\end{equation}
By first checking on $((M \tensor^{\sigma h} M) \tensor (X \tensor^h Y)) \times ((X^\ast \tensor^h Y^\ast) \tensor (N \tensor^{\sigma h} N))$ and then using separate weak$^\ast$-weak$^\ast$ continuity, we conclude that (\ref{diagram2}) commutes, which entails (\ref{faceq}) and thus completes the proof.
\end{proof}
\par 
Even though the Haagerup tensor product of two reflexive operator spaces need not be reflexive again, there are some nice consequences of Theorem \ref{vNthm}.
\begin{corollary} \label{colcor} 
Let $M$ and $N$ be injective von Neumann algebras. Then the set of those $\boldsymbol{x} \in M \bar{\tensor} N$ such that $\mathcal{T}_{M,N} \boldsymbol{x}$ factors completely boundedly through column Hilbert space is a subalgebra of $M \bar{\tensor} N$.
\end{corollary}
\begin{proof}
As the Haagerup tensor product of two column Hilbert spaces is again a column Hilbert space (\cite[Proposition 9.3.5]{ER}), the result follows immediately from Theorem \ref{vNthm}.
\end{proof}
\begin{remark}
By \cite[Proposition 9.3.5]{ER} and \cite[Corollary 2.12]{Pis}, respectively, analogous results hold for row Hilbert space and Pisier's operator Hilbert space.
\end{remark}
\par 
Following \cite{ER}, we denote the injective operator space tensor product by $\wTensor$; for $\cstar$-algebras, it is just the usual spatial tensor  product.
\begin{corollary} \label{modcor}
Let $M$ and $N$ be injective von Neumann algebras. Then $\mathcal{W}(M \bar{\tensor} N)$ is a bimodule over $M \wTensor N$.
\end{corollary}
\begin{proof}
Let $\boldsymbol{x} \in M \wTensor N$ and let $\boldsymbol{y} \in \mathcal{W}(M \bar{\tensor} N)$; we need to show that $\boldsymbol{xy}, \boldsymbol{yx} \in \mathcal{W}(M \bar{\tensor} N)$. There is no loss of generality to suppose that $\boldsymbol{x} \in M \tensor N$. By \cite{PSch} or \cite{Daw1}, $\mathcal{T}_{M,N} \boldsymbol{y}$ factors through a reflexive operator space, say $Y$, and trivially, $\mathcal{T}_{M,N} \boldsymbol{x}$ factors through a finite-dimensional operator space, say $X$. Consequently, $X \tensor^h Y$ and $Y \tensor^h X$ are reflexive. Hence, by Theorem \ref{vNthm}, $\mathcal{T}_{M,N}(\boldsymbol{xy})$ and $\mathcal{T}_{M,N}(\boldsymbol{yx})$ factor through $X \tensor^h Y$ and $Y \tensor^h X$, respectively, and, consequently, are weakly compact.
\end{proof}
\par
A von Neumann algebra $M$ is called \emph{subhomogeneous} if it has the form $M_{n_1}(A_1) \oplus \cdots \oplus M_{n_k}(A_k)$ with $n_1, \ldots, n_k \in \posints$ and abelian von Neumann algebras $A_1, \ldots, A_k$. If $M$ is subhomogeneous, the given operator space structure and $\min M$ are equivalent, i.e., the identity is completely bounded from $\min M$ to $M$.
\begin{corollary} \label{subhomcor}
Let $M$ and $N$ both be subhomogeneous von Neumann algebras. Then $\mathcal{W}(M \bar{\tensor} N)$ is a unital $\cstar$-subalgebra of $M \bar{\tensor} N$.
\end{corollary}
\begin{proof}
Let $\boldsymbol{x}, \boldsymbol{y} \in \mathcal{W}(M \bar{\tensor} N)$. Then $\mathcal{T}_{M,N} \boldsymbol{x}$ and $\mathcal{T}_{M,N} \boldsymbol{y}$ factor---according to \cite{Davetal}---through reflexive Banach spaces, say $X$ and $Y$, respectively. As the given operator structure of $M$ is equivalent to $\min M$ and the given operator structure on $N_\ast$ is equivalent to $\max N_\ast$, we see that $\mathcal{T}_{M,N} \boldsymbol{x}$ and $\mathcal{T}_{M,N} \boldsymbol{y}$ factor completely boundedly through $\min X$ and $\max Y$, respectively. As $(\min X) \tensor^h (\max Y)$ is reflexive, we conclude from Theorem \ref{vNthm} that $\mathcal{T}_{M,N}(\boldsymbol{xy})$ factors through $(\min X) \tensor^h (\max Y)$ and thus is weakly compact, i.e., $\boldsymbol{xy} \in \mathcal{W}(M \bar{\tensor} N)$.
\end{proof}
\section{Applications to Hopf--von Neumann algebras}
Recall the definition of a Hopf--von Neumann algebra.
\begin{definition}
A \emph{Hopf--von Neumann algebra} is a pair $(M,\Gamma)$ where $M$ is a von Neumann algebra, and $\Gamma \!: M \to M \bar{\tensor} M$ is a \emph{co-multiplication}, i.e., a faithful, normal, unital $^\ast$-homomorphism such that
\[
  (\Gamma \tensor \id) \circ \Gamma = (\id \tensor \Gamma) \circ \Gamma.
\]
\end{definition}
\par 
If $(M,\Gamma)$ is a Hopf--von Neumann algebra, then $M_\ast$ is a completely contractive Banach algebra in a canonical way through
\begin{equation} \label{product}
  \langle f \ast g, x \rangle := \langle f \tensor g, \Gamma x \rangle \qquad (f,g \in M_\ast, \, x \in M).
\end{equation}
\par
Hopf--von Neumann algebras arise naturally in abstract harmonic analysis:
\begin{example}
Let $G$ be a locally compact group. Define $\Gamma \!: L^\infty(G) \to L^\infty(G) \bar{\tensor} L^\infty(G) \cong L^\infty(G \times G)$ through
\[
  (\Gamma \phi)(x,y) := \phi(xy) \qquad (\phi \in L^\infty(G), \, x,y \in G).
\]
Then $(L^\infty(G),\Gamma)$ is a Hopf--von Neumann algebra, and the resulting product on $L^\infty(G)_\ast = L^1(G)$ is the usual convolution product. On the other hand, the co-multiplication
\[
  \hat{\Gamma} \!: \VN(G) \to \VN(G) \bar{\tensor} \VN(G), \quad \lambda(x) \mapsto \lambda(x) \tensor \lambda(x)
\]
($\lambda$ is here the left regular representation of $G$ on $L^2(G)$) yields the pointwise product on $A(G) = \VN(G)_\ast$.
\end{example}
\par 
Given a Banach algebra $A$, its dual space is a Banach $A$-bimodule in a canonical fashion. A functional $\phi \in A^\ast$ is called 
\emph{almost periodic} if the map
\begin{equation}  \label{lefty}
  A \to A^\ast, \quad a \mapsto a \cdot \phi
\end{equation}
is compact and \emph{weakly almost periodic} if is weakly compact. As
\begin{equation}  \label{righty}
  A \to A^\ast, \quad a \mapsto \phi \cdot a
\end{equation}
is only the adjoint of (\ref{lefty}) restricted to $A$, the perceived asymmetry in these definitions does, in fact, not exist. We set
\begin{align*}
  \AP(A) & := \{ \phi \in A^\ast : \text{$\phi$ is almost periodic} \} \\
  \intertext{and}
  \WAP(A) & := \{ \phi \in A^\ast : \text{$\phi$ is weakly almost periodic} \}.
\end{align*}
\par 
If $A$ is a completely contractive Banach algebra, the definitions of $\AP(A)$ and $\WAP(A)$ are somewhat unsatisfactory because they fail to take any operator space structure into account. 
\par 
In his Diplomarbeit \cite{Saa} under the supervision of G.\ Wittstock, H.\ Saar introduced the notion of a completely compact map between two operator spaces (or, rather, $\cstar$-algebras due to lack of abstract operator spaces at the time \cite{Saa} was written). In modern terminology, it reads as follows: for two operator spaces $E$ and $F$, a map $T \in \CB(E,F)$ is called \emph{completely compact} if, for each $\epsilon > 0$, there is a finite-dimensional subspace $Y_\epsilon$ such that $\| Q_{Y_\epsilon} T \|_\cb < \epsilon$, where $Q_{Y_\epsilon} \!: F \to F / Y_\epsilon$ is the quotient map. In \cite{Run2}, the author defined, for a completely contractive Banach algebra $A$, a functional $\phi \in A^\ast$ to be \emph{completely almost periodic} if the maps (\ref{lefty}) and (\ref{righty}) are completely compact (see \cite{Run2}, for a discussion of why we need to consider both (\ref{lefty}) and (\ref{righty}) here). We define
\[
  \CAP(A) := \{ \phi \in A^\ast : \text{$\phi$ is completely almost periodic} \}.
\]
\par
The following has been the primary motivation for the research in this paper:
\begin{question} 
Let $(M,\Gamma)$ be a Hopf--von Neumann algebra. Is $\WAP(M_\ast)$ a $\cstar$-subalgebra of $M$?
\end{question}
\par 
Of course, one can replace $\WAP(M_\ast)$ by $\AP(M_\ast)$ or $\CAP(M_\ast)$ in this question.
\par
The only difficult part in answering the question is to determine whether $\WAP(M_\ast)$---or $\AP(M_\ast)$ or $\CAP(M_\ast)$---is closed under multiplication. In \cite{Run2}, the author showed that $\CAP(M_\ast)$ is indeed a $\cstar$-algebra for injective $M$, and in \cite{Daw2}, Daws proved for abelian $M$ that both $\AP(M_\ast)$ and $\WAP(M_\ast)$ are $\cstar$-algebras.
\par 
We note:
\begin{proposition} \label{Wprop}
Let $(M,\Gamma)$ be a Hopf--von Neumann algebra. Then 
\[
  \WAP(M_\ast) = \{ x \in M : \mathcal{T}_{M,M}(\Gamma x) \in \mathcal{W}(M \bar{\tensor} N) \}.
\]
\end{proposition}
\begin{proof}
By (\ref{firstiso}) and (\ref{product}), we have
\[
  \mathcal{T}_{M,M} (\Gamma x) f = (f \tensor \id)(\Gamma x) = x \cdot f \qquad (x \in M, \, f \in M_\ast).
\]
The claim is then immediate.
\end{proof}
\par 
Let $(M,\Gamma)$ a Hopf--von Neumann algebra with $M$ injective. In \cite{Run2}, it was shown that
\[
  \CAP(M_\ast) = \{ x \in M : \Gamma x \in M \wTensor M \},
\]
(which immediately yields that $\CAP(M_\ast)$ is a $\cstar$-algebra). In view of Proposition \ref{Wprop}, we thus obtain from Corollary \ref{modcor}:
\begin{corollary}
Let $(M,\Gamma)$ be a Hopf--von Neumann algebra such that $M$ is injective. Then $\WAP(M_\ast)$ is a bimodule over $\CAP(M_\ast)$.
\end{corollary}
\begin{remark}
This applies, for instance, to $(\VN(G),\hat{\Gamma})$ where $G$ is amenable or connected.
\end{remark}
\par 
Via Proposition \ref{Wprop}, we also recover (and mildly improve) \cite[Theorem 4]{Daw2} from Corollary \ref{subhomcor}:
\begin{corollary}
Let $(M,\Gamma)$ be a Hopf--von Neumann algebra such that $M$ is subhomogeneous. Then $\WAP(M_\ast)$ is a $\cstar$-subalgebra of $M$.
\end{corollary}
\begin{remark}
The analogous statement is also true for $\AP(M_\ast)$: the proof of \cite[Theorem 1]{Daw2} adapts effortlessly from the abelian to the general subhomogeneous case.
\end{remark}
\par 
Let $(M,\Gamma)$ be a Hopf--von Neumann algebra with $M$ injective, and let $x, y \in M$ be such that $\mathcal{T}_{M,M}(\Gamma x)$ factors through a minimal operator space and $\mathcal{T}_{M,M}(\Gamma y)$ factors through a maximal one, then $xy \in \WAP(M_\ast)$ by Theorem \ref{vNthm} and Proposition \ref{Wprop}. If both $\mathcal{T}_{M,M}(\Gamma x)$ and $\mathcal{T}_{M,M}(\Gamma y)$ factor through a minimal operator space, then we generally don't know if $xy \in \WAP(M_\ast)$: at the Banach space level, the Haagerup tensor product of two minimal operator spaces is just Grothendieck's $H$-tensor product (\cite[Proposition 4.1]{BP}), which doesn't preserve reflexivity.
\par 
Surprisingly, for certain Hopf--von Neumann algebras, however, there is a way around this. Let $(M,\Gamma)$ be a Hopf--von Neumann algebra, and let $\Sigma \!: M \bar{\tensor} M \to M \bar{\tensor} M$ be the normal extension of the flip map $M \tensor M \ni x \tensor y \mapsto y \tensor x$. We call $(M,\Gamma)$ \emph{co-commutative} if $\Sigma \Gamma = \Gamma$, which is equivalent to $M_\ast$ being commutative. For instance, $(\VN(G),\hat{\Gamma})$ is co-commutative for every locally compact group $G$.
\begin{proposition} \label{cocom}
Let $(M,\Gamma)$ be a co-commutative von Neumann algebra with $M$ injective, and let $x,y \in M$ be such that $\mathcal{T}_{M,M}(\Gamma x)$ and $\mathcal{T}_{M,M}(\Gamma y)$ each factor through a minimal reflexive operator space. Then $xy \in \WAP(M_\ast)$.
\end{proposition}
\begin{proof}
There are reflexive Banach spaces $X$ and $Y$ such that $\mathcal{T}_{M,M}(\Gamma x)$ and $\mathcal{T}_{M,M}(\Gamma y)$ factor through $\min X$ and $\min Y$, respectively. Consequently, $\mathcal{T}_{M,M}(\Gamma y)^\ast \in \CB(M^\ast,M)$ factors completely boundedly through $\max Y^\ast$, as does its restriction to $M_\ast$. Note that, by (\ref{firstiso}) and the co-commutativity of $(M,\Gamma)$, we have
\[
  \mathcal{T}_{M,M}(\Gamma y)^\ast f = (\id \tensor f)(\Gamma y) = (\id \tensor f)(\Sigma \Gamma y) = (f \tensor \id)(\Gamma y) = \mathcal{T}_{M,M}(\Gamma y)f \qquad (f \in M_\ast),
\]
so that $\mathcal{T}_{M,M}(\Gamma y)$ factors through $\max Y^\ast$. As $(\min X) \tensor^h (\max Y^\ast)$ is reflexive, Theorem \ref{vNthm} asserts that $\mathcal{T}_{M,M}(\Gamma(xy)) = \mathcal{T}_{M,M}((\Gamma x)(\Gamma y))$ factors through $(\min X) \tensor^h (\max Y^\ast)$ and thus is weakly compact, i.e., $xy \in \WAP(M_\ast)$.
\end{proof}
\begin{remarks}
\item Of course, the conclusion of Theorem \ref{cocom} remains true if we require instead that $\mathcal{T}_{M,M}(\Gamma x)$ and $\mathcal{T}_{M,M}(\Gamma y)$ both factor completely boundedly through a maximal reflexive operator space.
\item We do not know if $\mathcal{T}_{M,M}(\Gamma(xy))$ factors again through a minimal operator space.
\end{remarks}
\renewcommand{\baselinestretch}{1.0}
\renewcommand{\baselinestretch}{1.2}
\dated
\vfill
\begin{tabbing}
\textit{Author's address}: \= Department of Mathematical and Statistical Sciences \\
\> University of Alberta \\
\> Edmonton, Alberta \\
\> Canada T6G 2G1 \\[\medskipamount]
\textit{E-mail}: \> \texttt{vrunde@ualberta.ca} \\[\medskipamount]
\textit{URL}: \> \texttt{http://www.math.ualberta.ca/$^\sim$runde/}
\end{tabbing}          
\end{document}